\documentclass[10pt]{amsart}
\usepackage{amscd}
\usepackage{amssymb}
\usepackage[all]{xy}
\usepackage{lmodern} 
\usepackage[T1]{fontenc}
\date{11 December 2012}
\title{Completion by Derived Double Centralizer}
\usepackage{hyperref}
\hypersetup{colorlinks=false}

\author{Marco Porta, Liran Shaul and Amnon Yekutieli}

\address{Department of  Mathematics,
Ben Gurion University, Be'er Sheva 84105, Israel}

\address{ {\em E-mail address}:  (Porta) {\tt marcoporta1@libero.it},
(Shaul) {\tt shlir@math.bgu.ac.il},
(Yekutieli) {\tt amyekut@math.bgu.ac.il} } 

\thanks{{\em Mathematics Subject Classification} 2010.
Primary: 13D07; Secondary: 13B35, 13C12, 13D09, 18E30.}
\keywords{Adic completion, derived functors, derived Morita theory.}
\thanks{This research was supported by the Israel Science Foundation and the 
Center for Advanced Studies at BGU}

\newtheorem{thm}[equation]{Theorem}
\newtheorem{cor}[equation]{Corollary}
\newtheorem{prop}[equation]{Proposition}
\newtheorem{lem}[equation]{Lemma}
\theoremstyle{definition}
\newtheorem{dfn}[equation]{Definition}
\newtheorem{rem}[equation]{Remark}

\numberwithin{equation}{section}

\newcommand{\iso}{\xrightarrow{\simeq}}

\newcommand{\xar}{\xrightarrow}
\newcommand{\opn}{\operatorname}
\newcommand{\cat}[1]{\operatorname{\mathsf{#1}}}

\newcommand{\mfrak}[1]{\mathfrak{#1}}

\newcommand{\mrm}[1]{\mathrm{#1}}
\newcommand{\mbb}[1]{\mathbb{#1}}

\newcommand{\tup}[1]{\textup{#1}}
\newcommand{\bsym}[1]{\boldsymbol{#1}}
\newcommand{\boplus}{\bigoplus\nolimits}

\newcommand{\til}[1]{\tilde{#1}}
\newcommand{\what}[1]{\widehat{#1}}

\newcommand{\K}{\mbb{K}}

\newcommand{\N}{\mbb{N}}
\newcommand{\Z}{\mbb{Z}}

\newcommand{\La}{\Lambda}
\newcommand{\Ga}{\Gamma}
\renewcommand{\a}{\mfrak{a}}

\renewcommand{\d}{\mrm{d}}
\newcommand{\ot}{\otimes}

\newcommand{\sbmat}[1]{\left[ \begin{smallmatrix} #1
\end{smallmatrix} \right]}

\newcommand{\lb}{\linebreak}
\newcommand{\distri}{\xar{ \vartriangle }}
\renewcommand{\ot}{\otimes}

\begin{document}

\begin{abstract}
Let $A$ be a commutative ring, and let $\a$ be a {\em weakly proregular} ideal
in $A$. (If $A$ is noetherian then any ideal in it is weakly proregular.)
Suppose $M$ is a compact generator of the category of {\em cohomologically
$\a$-torsion complexes}. We prove that the {\em  derived double centralizer} of
$M$ is isomorphic to the $\a$-adic completion of $A$. The proof relies on
the {\em MGM equivalence} from \cite{PSY} and on {\em derived Morita
equivalence}. Our result extends earlier work of Dwyer-Greenlees-Iyengar
\cite{DGI} and Efimov \cite{Ef}.
\end{abstract}

\maketitle

\setcounter{section}{-1}
\section{Introduction}

Let $A$ be a commutative ring. We denote by 
$\cat{D}(\cat{Mod} A)$ the derived
category of $A$-modules. Given $M \in \cat{D}(\cat{Mod} A)$ we define 
\begin{equation} \label{equ:520}
\opn{Ext}_A(M) 
:= \bigoplus_{i \in \Z} \, \opn{Hom}_{\cat{D}(\cat{Mod} A)}(M, M[i]) .
\end{equation}
This is a graded $A$-algebra with the Yoneda multiplication,
which we call the {\em Ext algebra} of $M$. 

Suppose we choose a K-projective resolution $P \to M$. The resulting DG
$A$-algebra $B := \opn{End}_A(P)$ is called a {\em derived endomorphism DG
algebra} of $M$. It turns out (see Proposition \ref{prop:17}) that 
the DG algebra $B$ is unique up to quasi-isomorphism; and of course its
cohomology $\opn{H}(B) :=  \bigoplus_{i \in \Z} \, \opn{H}^i(B)$ 
is canonically
isomorphic to $\opn{Ext}_A(M)$ as graded $A$-algebra.

Consider the derived category $\til{\cat{D}}(\cat{DGMod} B)$
of left DG $B$-modules. We can view $P$ as an object of 
$\til{\cat{D}}(\cat{DGMod} B)$, and thus, like in (\ref{equ:520}), we get
the graded $A$-algebra $\opn{Ext}_B(P)$.
This is the {\em derived double centralizer algebra} of $M$. By Corollary
\ref{cor:300}, the graded algebra  $\opn{Ext}_B(P)$ is independent of the
resolution  $P \to M$, up to isomorphism.

Let $\a$ be an ideal in $A$. The $\a$-torsion functor $\Ga_{\a}$ can be right
derived, giving a triangulated functor 
$\mrm{R} \Ga_{\a}$ from $\cat{D}(\cat{Mod} A)$ to itself.
A complex $M \in \cat{D}(\cat{Mod} A)$ is called a {\em cohomologically
$\a$-torsion complex} if the canonical morphism 
$\mrm{R} \Ga_{\a}(M) \to M$ is an isomorphism. The full triangulated category
on the cohomologically torsion complexes is denoted by 
$\cat{D}(\cat{Mod} A)_{\a \tup{-tor}}$.
It is known that when $\a$ is finitely generated, the category 
$\cat{D}(\cat{Mod} A)_{\a \tup{-tor}}$ is compactly generated (for instance by
the Koszul complex $\opn{K}(A; \bsym{a})$ associated to a finite generating
sequence $\bsym{a} = (a_1, \ldots, a_n)$ of $\a$). 

Let us denote by $\what{A}$ the $\a$-adic completion of $A$. This is a
commutative $A$-algebra, and in it there is the ideal 
$\what{\a} := \a \cdot \what{A}$. If $\a$ is finitely generated, then the ring 
$\what{A}$ is $\what{\a}$-adically complete.

A {\em weakly proregular sequence} in $A$ is a finite sequence 
$\bsym{a}$ of elements of $A$, whose 
Koszul cohomology satisfies certain vanishing conditions; 
see Definition \ref{dfn:250}. This concept was introduced by
Alonso-Jeremias-Lipman \cite{AJL} and Schenzel \cite{Sc}. An ideal $\a$ in $A$
is called {\em weakly proregular} if it can be generated by a weakly proregular
sequence. It is important to note that if $A$ is noetherian, then any finite
sequence in it is weakly proregular, so that any ideal in $A$ is weakly
proregular. But there are some fairly natural non-noetherian examples (see
\cite[Example 3.0(b)]{AJL} and \cite[Example 4.35]{PSY}). 

Here is our main result (repeated as Theorem \ref{thm:22} in the body of the
paper).

\begin{thm} \label{thm:53}
Let $A$ be a commutative ring, let  $\a$ be a weakly proregular ideal in $A$,
and let $M$ be a compact generator of $\cat{D}(\cat{Mod} A)_{\a \tup{-tor}}$.
Choose a K-projective resolution $P \to M$, and define 
$B := \opn{End}_A(P)$. Then there is a unique isomorphism
of graded $A$-algebras
$\opn{Ext}_{B}(P) \cong \what{A}$.
\end{thm}

Our result extends earlier work of Dwyer-Greenlees-Iyengar \cite{DGI}
and Efimov \cite{Ef}; see Remark \ref{rem:1} for a discussion.

Let us say a few words on the proof of Theorem \ref{thm:53}. 
We use {\em derived Morita theory} to find an isomorphism of graded algebras
between  $\opn{Ext}_B(P)$ and $\opn{Ext}_A(N)^{\mrm{op}}$, where 
$N :=  \mrm{R} \Ga_{\a}(A)$. The necessary facts about derived Morita theory
are recalled in Section \ref{sec:der-morita}. We then use {\em MGM equivalence}
(recalled in Section \ref{sec:kosz}) to prove that 
$\opn{Ext}_A(N) \cong \opn{Ext}_A(\what{A}) \cong \what{A}$.

\medskip \noindent
\textbf{Acknowledgments.}
We wish to thank Bernhard Keller, John Greenlees, Alexander Efimov, Maxim
Kontsevich, Vladimir Hinich and Peter J{\o}rgensen for helpful discussions.
We are also grateful to the anonymous referee for a careful reading of the
paper and constructive remarks.

\section{Weak Proregularity and MGM Equivalence} \label{sec:kosz}

Let $A$ be a commutative ring, and let $\a$ be an ideal in it. 
(We do not assume that $A$ is noetherian or $\a$-adically complete.)
There are two operations on $A$-modules associated to this data: the {\em
$\a$-adic completion} and the {\em $\a$-torsion}. 
For an $A$-module $M$ its $\a$-adic completion is the $A$-module
$\La_{\a} (M) = \what{M} := \lim_{\leftarrow i} \, M / \a^i M$.
An element $m \in M$ is called an $\a$-torsion element if
 $\a^i m = 0$ for $i \gg 0$. The $\a$-torsion elements form 
the $\a$-torsion submodule $\Gamma_{\a} (M)$ of $M$.

Let us denote by $\cat{Mod} A$ the category of $A$-modules. So we have additive
functors $\Lambda_{\a}$ and $\Gamma_{\a}$ from $\cat{Mod} A$ to itself. 
The functor $\Gamma_{\a}$ is left exact; whereas $\Lambda_{\a}$ is neither left
exact nor right exact. An $A$-module is called {\em $\a$-adically complete} if
the canonical homomorphism $\tau_M : M \to \La_{\a}(M)$ is bijective
(some texts would say that $M$ is complete and separated); and $M$
is {\em $\a$-torsion} if the canonical homomorphism 
$\sigma_M : \Ga_{\a}(M) \to M$ is bijective.
If the ideal $\a$ is finitely generated, then the functor $\La_{\a}$ is
idempotent; namely for any module $M$, its completion $\La_{\a}(M)$ is
$\a$-adically complete. (There are counterexamples to that for infinitely
generated ideals -- see \cite[Example 1.8]{Ye}.)

The derived category of $\cat{Mod} A$ is denoted by $\cat{D}(\cat{Mod} A)$. 
The derived functors
\[ \mrm{L} \Lambda_{\a}, \mrm{R} \Gamma_{\a} : 
\cat{D}(\cat{Mod} A) \to \cat{D}(\cat{Mod} A) \]
exist. The left derived functor $\mrm{L} \Lambda_{\a}$ is constructed using
K-flat resolutions, and the right derived functor $\mrm{R} \Gamma_{\a}$
is constructed using K-injective resolutions.
This means that for any K-flat complex $P$, the canonical morphism 
$\xi_P : \mrm{L} \Lambda_{\a}(P) \to \Lambda_{\a}(P)$
is an isomorphism; and for any K-injective complex $I$, the canonical morphism 
$\xi_I : \Gamma_{\a}(I) \to \mrm{R} \Gamma_{\a}(I)$
is an isomorphism.
The relationship between the derived functors $\mrm{R} \Gamma_{\a}$ and 
$\mrm{L} \Lambda_{\a}$ was first studied in \cite{AJL}, where the {\em
Greenlees-May duality} was established (following the paper \cite{GM}).

A complex $M \in \cat{D}(\cat{Mod} A)$ is called a {\em cohomologically
$\a$-torsion complex} if the canonical morphism 
$\sigma^{\mrm{R}}_M : \mrm{R} \Gamma_{\a} (M) \to M$
is an isomorphism.
The complex $M$  is called a {\em cohomologically $\a$-adically complete
complex} if the canonical morphism 
$\tau^{\mrm{L}}_M : M \to \mrm{L} \Lambda_{\a} (M)$
is an isomorphism.
We denote by $\cat{D}(\cat{Mod} A)_{\a \tup{-tor}}$ and 
$\cat{D}(\cat{Mod} A)_{\a \tup{-com}}$
the full subcategories of $\cat{D}(\cat{Mod} A)$ consisting of
cohomologically $\a$-torsion complexes and cohomologically $\a$-adically
complete complexes, respectively. These are triangulated subcategories.

Very little can be said about the functors 
$\mrm{L} \Lambda_{\a}$ and $\mrm{R}\Gamma_{\a}$, and about the corresponding
triangulated categories $\cat{D}(\cat{Mod} A)_{\a \tup{-tor}}$ and 
$\cat{D}(\cat{Mod} A)_{\a \tup{-com}}$, in general. However we know a lot 
when the ideal $\a$ is {\em weakly proregular}. 

Before defining weak proregularity we have to talk about {\em Koszul complexes}.
Recall that for an element $a \in A$ the Koszul complex is
\[ \mrm{K}(A; a) :=  \bigl( \cdots \to 0 \to A \xar{a \cdot} A \to 0 \to
\cdots \bigr) , \]
concentrated in degrees $-1$ and $0$. 
Given a finite sequence $\bsym{a} = (a_1, \ldots, a_n)$ of elements of $A$, 
the Koszul complex associated to this sequence is 
\[ \opn{K}(A; \bsym{a}) := 
\opn{K}(A; a_1) \otimes_A \cdots \otimes_A \opn{K}(A; a_n) . \]
This is a complex of finitely generated free $A$-modules, concentrated in
degrees $-n, \ldots, 0$. 
There is a canonical isomorphism of $A$-modules
$\mrm{H}^0 (\opn{K}(A; \bsym{a})) \cong A / (\bsym{a})$,
where $(\bsym{a})$ is the ideal generated by the sequence $\bsym{a}$. 

For any $i \geq 1$ let $\bsym{a}^i := (a_1^i, \ldots, a_n^i)$.
If $j \geq i$ then there is a canonical homomorphism of complexes 
$p_{j,i} : \opn{K}(A; \bsym{a}^j) \to \opn{K}(A; \bsym{a}^i)$,
which in $\mrm{H}^0$ corresponds to the surjection 
$A / (\bsym{a}^j) \to A / (\bsym{a}^i)$.
Thus for every $k \in \Z$ we get an inverse system of $A$-modules 
\begin{equation} \label{eqn:200}
\bigl\{ \mrm{H}^k ( \opn{K}(A; \bsym{a}^{i}) ) \bigr\}_{i \in \N} \ ,
\end{equation}
with transition homomorphisms 
\[ \mrm{H}^k (p_{j,i} ) : \mrm{H}^k ( \opn{K}(A; \bsym{a}^{j} )) \to
\mrm{H}^k ( \opn{K}(A; \bsym{a}^{i}) ) . \]
Of course for $k = 0$ the inverse limit equals the $(\bsym{a})$-adic completion
of $A$. What turns out to be crucial is the behavior of this inverse system for
$k < 0$. For more details please see \cite[Section 4]{PSY}.

An inverse system 
$\{ M_i \}_{i \in \N}$
of abelian groups, with transition maps $p_{j, i} :M_j \to M_i$,
is called {\em pro-zero} if for every $i$ there exists $j \geq i$ such that 
$p_{j, i}$ is zero.

\begin{dfn} \label{dfn:250} 
\begin{enumerate}
\item Let $\bsym{a}$ be a finite sequence in $A$. The sequence 
$\bsym{a}$ is called  a {\em weakly proregular sequence}  if for every 
$k \leq -1$ the inverse system (\ref{eqn:200}) is pro-zero.

\item An ideal $\a$ in $A$ is called a {\em a weakly proregular ideal} if
it is generated by some weakly proregular sequence. 
\end{enumerate}
\end{dfn}

The etymology of the name ``weakly proregular sequence'', and the history
of related concepts, are explained in \cite{AJL} and \cite{Sc}.

If $\bsym{a}$ is a regular sequence, then it is weakly proregular. 
More important is the following result. 

\begin{thm}[\cite{AJL}] \label{thm:500}
If $A$ is noetherian, then every finite sequence in $A$ is
weakly proregular, so that every ideal in $A$ is weakly proregular.
\end{thm}

Here is another useful fact. 

\begin{thm}[\cite{Sc}] \label{thm:501}
Let $\a$ be a weakly proregular ideal in a ring $A$. Then any finite sequence
that generates $\a$ is weakly proregular.
\end{thm}

These theorems are repeated (with different proofs)
as \cite[Theorem 4.34]{PSY} and \cite[Corollary 6.3]{PSY} respectively.

As the next theorem shows, weak proregularity is the correct condition for the
derived torsion functor to be ``well-behaved''.
Suppose $\bsym{a}$ is a finite  sequence that generates the ideal $\a \subset
A$. Consider the infinite dual Koszul complex 
\[ \opn{K}_{\infty}^{\vee}(A; \bsym{a}) := 
\lim_{i \to} \opn{Hom}_A \bigl( \opn{K}(A; \bsym{a}^i) , A \bigr) . \]
Given a complex $M$, there is a canonical morphism 
\begin{equation} \label{eqn:455}
\mrm{R} \Gamma_{\a} (M) \to 
\opn{K}_{\infty}^{\vee}(A; \bsym{a}) \ot_A M
\end{equation}
in $\cat{D}(\cat{Mod} A)$. 

\begin{thm}[\cite{Sc}] \label{thm:503}
The sequence $\bsym{a}$ is weakly proregular iff the morphism
\tup{(\ref{eqn:455})} is an isomorphism for every $M \in \cat{D}(\cat{Mod} A)$. 
\end{thm}

The following theorem, which is \cite[Theorem 1.1]{PSY}, plays a central role
in our work.

\begin{thm}[MGM Equivalence] \label{thm:201}
Let $A$ be a commutative ring, and $\a$ a weakly proregular ideal in it.
\begin{enumerate}
\item For any $M \in \cat{D}(\cat{Mod} A)$ one has
$\mrm{R} \Gamma_{\a} (M) \in \cat{D}(\cat{Mod} A)_{\a \tup{-tor}}$
and 
$\mrm{L} \Lambda_{\a} (M) \in \cat{D}(\cat{Mod} A)_{\a \tup{-com}}$.

\item The functor 
\[ \mrm{R} \Gamma_{\a} : 
\cat{D}(\cat{Mod} A)_{\a \tup{-com}} \to 
\cat{D}(\cat{Mod} A)_{\a \tup{-tor}} \]
is an equivalence, with quasi-inverse $\mrm{L} \Lambda_{\a}$. 
\end{enumerate}
\end{thm}

\begin{rem}
Slightly weaker versions of Theorem \ref{thm:201} appeared previously; they are
\cite[Theorem (0.3)$^*$]{AJL} and \cite[Theorem 4.5]{Sc}.
The difference is that in these earlier results it was assumed that the ideal
$\a$ is generated by a sequence $(a_1, \ldots, a_n)$ that is weakly proregular,
and moreover each $a_i$ has bounded torsion. This extra condition certainly
holds when $A$ is noetherian. 

For the sake of convenience, in the present paper we quote \cite{PSY} regarding
derived completion and torsion. It is tacitly understood that in the
noetherian case the results of \cite{AJL} and \cite{Sc} suffice.
\end{rem}

\section{The Derived Double Centralizer} \label{sec:der-end}

In this section we define the derived double centralizer of a DG module. See
Remarks \ref{rem:401} and \ref{rem:400} for a discussion of this concept and
related literature.

Let $\K$ be a commutative ring, and let $A = \bigoplus_{i \in \Z} A^i$
be a DG $\K$-algebra (associative and unital, but not necessarily commutative).
Given left DG  $A$-modules $M$ and $N$, we denote by 
$\opn{Hom}^i_{A}(M, N)$ the $\K$-module of  $A$-linear homomorphisms of degree
$i$. We get a DG $\K$-module 
\[ \opn{Hom}_{A}(M, N) := \bigoplus_{i \in \Z} \, \opn{Hom}^i_{A}(M, N)  \]
with the usual differential.

The object 
$\opn{End}_{A}(M) := \opn{Hom}_{A}(M, M)$
is a DG $\K$-algebra.
Since the left actions of $A$ and $\opn{End}_{A}(M)$ on $M$ commute, we see
that $M$ is a left DG module over the DG algebra
$A \ot_{\K} \opn{End}_{A}(M)$. 

The category of left DG $A$-modules is denoted by $\cat{DGMod} A$.
The set of morphisms $\opn{Hom}_{\cat{DGMod} A}(M, N)$ is 
precisely the set of $0$-cocycles in the DG $\K$-module $\opn{Hom}_A(M, N)$.
Note that $\cat{DGMod} A$ is a $\K$-linear abelian category.

Let $\til{\cat{K}}(\cat{DGMod} A)$ be the homotopy category of
$\cat{DGMod} A$, so that 
\[ \opn{Hom}_{\til{\cat{K}}(\cat{DGMod} A)}(M, N) = 
\opn{H}^0 \bigl( \opn{Hom}_A(M, N) \bigr) . \]
The derived category $\til{\cat{D}}(\cat{DGMod} A)$ is gotten by inverting 
the quasi-isomorphisms in $\til{\cat{K}}(\cat{DGMod} A)$.
The categories $\til{\cat{K}}(\cat{DGMod} A)$  and 
$\til{\cat{D}}(\cat{DGMod}A)$ are $\K$-linear and triangulated.
If $A$ happens to be a ring (i.e.\ $A^i = 0$ for $i \neq 0$) then 
$\cat{DGMod} A = \cat{C}(\cat{Mod} A)$,
the category of complexes  in $\cat{Mod} A$, and 
$\til{\cat{D}}(\cat{DGMod} A) = \cat{D}(\cat{Mod} A)$,
the usual derived category.

For $M, N \in \cat{DGMod}A$ we define 
\[ \opn{Ext}^i_A(M, N) := \opn{Hom}_{\til{\cat{D}}(\cat{DGMod} A)}(M, N[i]) \]
and 
\[ \opn{Ext}_A(M, N) := \bigoplus_{i \in \Z} \, \opn{Ext}^i_A(M, N) . \]

\begin{dfn} \label{dfn:300}
Let $A$ be a DG $\K$-algebra and $M \in \cat{DGMod}A$.
Define 
\[ \opn{Ext}_A(M) := \opn{Ext}_A(M, M) . \]
This is a graded $\K$-algebra with the Yoneda multiplication (i.e.\ composition
of morphisms in $\til{\cat{D}}(\cat{DGMod} A)$). We call $\opn{Ext}_A(M)$ the
{\em Ext algebra} of $M$.
\end{dfn}

There is a canonical homomorphism of graded $\K$-algebras 
$\opn{H}(\opn{End}_A(M)) \to \lb \opn{Ext}_A(M)$. If $M$ is either K-projective
or K-injective, then this homomorphism is bijective. 

\begin{dfn} \label{dfn:310}
Let $A$ be a DG $\K$-algebra and $M$ a DG $A$-module. Choose a K-projective
resolution $P \to M$ in $\cat{DGMod} A$. The DG $\K$-algebra 
$B := \opn{End}_A(P)$ is called a {\em derived endomorphism DG algebra} of $M$.
\end{dfn}

Note that there are isomorphisms of graded $\K$-algebras
$\opn{H}(B) \cong \opn{Ext}_A(P) \cong \opn{Ext}_A(M)$.
The dependence of the derived endomorphism DG algebra $B = \opn{End}_A(P)$ on
the resolution $P \to M$ is explained in the next proposition. 

\begin{prop} \label{prop:17}
Let $M$ be a DG $A$-module, and let $P \to M$ and $P' \to M$ be 
K-projective resolutions in $\cat{DGMod} A$. Define
$B := \opn{End}_{A}(P)$ and $B' := \opn{End}_{A}(P')$. Then there is a 
DG $\K$-algebra $B''$, and a DG $B''$-module $P''$, 
with DG $\K$-algebra quasi-isomorphisms
$B'' \to B$ and $B'' \to B'$, and DG $B''$-module quasi-iso\-mor\-phisms
$P'' \to P$ and $P'' \to P'$. 
\end{prop}

\begin{proof}
Choose a quasi-isomorphism $\phi : P' \to P$ in 
$\cat{DGMod} A$ lifting the given quasi-isomorphisms to
$M$; this can be done of course. Let 
$L := \opn{cone}(\phi) \in \lb \cat{DGMod}A$, the mapping cone of $\phi$. So
as graded $A$-module 
$L =  P \oplus P'[1] = \sbmat{P \\ P'[1]}$; and the differential is
$\d_L = \sbmat{\d_P & \phi \\ 0 & \d_{P'[1]}}$, where $\phi$ is viewed as a
degree $1$ homomorphism $P'[1] \to P$. Of course $L$ is an acyclic DG module.

Take $Q := \opn{Hom}_A(P'[1], P)$, and let $B''$ be the triangular
matrix graded algebra
$B'' := \sbmat{B & Q \\ 0 & B'}$
with the obvious matrix multiplication.
This makes sense because there is a canonical isomorphism of DG algebras
$B' \cong \opn{End}_{A}(P'[1])$. 
Note that $B''$ is a subalgebra of $\opn{End}_{A}(L)$.
We make $B''$ into a DG algebra with differential 
$\d_{B''} := \d_{\opn{End}_{A}(L)}|_{B''}$.
The projections $B'' \to B$ and $B'' \to B'$ on the diagonal entries are DG
algebra  quasi-isomorphisms, because their kernels are the acyclic complexes 
$\opn{Hom}_{A}(P'[1], L)$ and $\opn{Hom}_{A}(L, P)$ respectively. 

Now under the restriction functor 
$\cat{DGMod}(B) \to \cat{DGMod}(B'')$ we have 
$P \mapsto \sbmat{P \\ 0}$, and likewise
$P' \mapsto \sbmat{0 \\ P'}$. 
Consider the exact sequence 
\[ 0 \to  \sbmat{P \\ 0} \to L \to \sbmat{0 \\ P'[1]} \to 0 \]
in $\cat{DGMod} (B'')$. 
There is an induced distinguished triangle 
$\sbmat{0 \\ P'} \xar{\chi} \sbmat{P \\ 0} \to L \distri$
in $\til{\cat{D}}(\cat{DGMod} (B''))$. But $L$ is acyclic, so $\chi$ is an
isomorphism.

Finally let us choose a K-projective resolution 
$P'' \to \sbmat{0 \\ P'}$ in $\cat{DGMod} (B'')$. 
Then $\chi$ induces a quasi-isomorphism 
$P'' \to \sbmat{P \\ 0}$ in $\cat{DGMod} (B'')$. 
\end{proof}

\begin{cor} \label{cor:300}
In the situation of Proposition \tup{\ref{prop:17}}, there is an isomorphism of
graded $\K$-algebras 
\[ \opn{Ext}_B(P) \cong \opn{Ext}_{B'}(P') . \]
\end{cor}

\begin{proof}
Since $B'' \to B$ is a quasi-isomorphism of DG algebras, it follows that the
restriction functor
$\til{\cat{D}}( \cat{DGMod}(B)) \to 
\til{\cat{D}}(\cat{DGMod} (B''))$
is an equivalence of triangulated categories. Therefore we get an
induced  isomorphism of graded $\K$-algebras 
$\opn{Ext}_B(P) \iso \opn{Ext}_{B''}(P'')$. 
Similarly we get a graded $\K$-algebra isomorphism
$\opn{Ext}_{B'}(P') \iso \opn{Ext}_{B''}(P'')$. 
\end{proof}

\begin{dfn} \label{dfn:303}
Let $M$ be a DG $A$-module, and let $P \to M$ be a K-projective resolution in 
$\cat{DGMod}A$. The graded $\K$-algebra $\opn{Ext}_{B}(P)$ is called a 
{\em derived double centralizer} of $M$. 
\end{dfn}

\begin{rem} \label{rem:401}
The uniqueness of the graded $\K$-algebra $\opn{Ext}_{B}(P)$ provided by
Corollary \ref{cor:300} is sufficient for the purposes of this paper (see
Theorem \ref{thm:53}). 

It is possible to show by a more detailed calculation that the isomorphism
provided by Corollary \ref{cor:300} is in fact canonical (it does not depend on
the choices made in the proof of Proposition \ref{prop:17}, e.g.\ the
quasi-isomorphism $\phi$). 

Let us choose a K-projective resolution $Q \to P$ in $\cat{DGMod}(B)$, and
define the DG algebra $C := \opn{End}_B(Q)$. Then $C$ should be called a
{\em double endomorphism DG algebra} of $M$. Of course 
$\opn{H}(C) \cong \opn{Ext}_{B}(P)$ as graded algebras. There should be a
canonical DG algebra homomorphism $A \to C$. 

We tried to work out a comprehensive treatment of derived
endomorphism algebras and their iterates, using the old-fashioned methods,
and did not get very far (hence it is not included in the paper). We
expect that a full treatment is only possible in terms of
$\infty$-categories. 
\end{rem}

\begin{rem} \label{rem:400}
Derived endomorphism DG algebras (and the double derived ones) were treated in
several earlier the papers, including \cite{DGI}, \cite{Jo} and \cite{Ef}.
These papers do not mention any uniqueness properties of these DG algebras;
indeed, as far as we can tell, they just pick a convenient resolution $P \to M$,
and work with the DG algebra $\opn{End}_{A}(P)$. Cf.\ Subsection 1.5 of
\cite{DGI} where this issue is briefly discussed.

The most detailed treatment of derived endomorphism DG algebras that we know is
in Keller's paper \cite{Ke}. In \cite[Section 7.3]{Ke} the concept of a {\em
lift} of a DG module is introduced. The pair $(B, P)$ from Definition
\ref{dfn:310} is called a {\em standard lift} in \cite{Ke}. It is proved that
lifts are unique up to quasi-isomorphism (this is basically what is done in our
Proposition \ref{prop:17}); but there is no statement regarding uniqueness of
these quasi-isomorphisms. Also there is no discussion of derived double
centralizers.
\end{rem}

\section{Supplement on Derived Morita Equivalence} \label{sec:der-morita} 

Derived Morita theory goes back to Rickard's paper [Ri], which dealt with rings
and two-sided tilting complexes. 
Further generalizations can be found in \cite{Ke, BV, Jo}.
For our purposes (in Section \ref{sec:ddc}) we need to know certain precise
details about derived Morita equivalence in the case of DG algebras and compact
generators (specifically, formula (\ref{eqn:451})
for the functor $F$ appearing in Theorem \ref{thm:34}); and
hence we give the full proof here. 

Let $\cat{E}$ be a triangulated category with infinite direct sums. 
Recall that an object $M \in \cat{E}$ is called {\em compact} (or small) if 
for any collection $\{ N_z \}_{z \in Z}$ of objects of $\cat{E}$, the 
canonical homomorphism
\[ \bigoplus_{z \in Z} \, \opn{Hom}_{\cat{E}}(M, N_z) \to 
\opn{Hom}_{\cat{E}} \Bigl( M, \boplus_{z \in Z} \, N_z
\Bigr) \]
is bijective.
The object $M$ is called {\em generator of $\cat{E}$} if for any nonzero
object $N \in \cat{E}$ there is some $i \in \Z$ such that 
$\opn{Hom}_{\cat{E}}(M, N[i]) \neq 0$.

As in Section \ref{sec:der-end} we consider a commutative ring $\K$ and a DG
$\K$-algebra $A$. The next lemma seems to be known, but we could not find
a reference.

\begin{lem} \label{lem:270}
Let $\cat{E}$ be a triangulated category with infinite direct sums, let 
$F, G : \til{\cat{D}}(\cat{DGMod} A) \to \cat{E}$ be
triangulated functors that commute with infinite direct sums,
and let $\eta : F \to G$
be a morphism of triangulated functors. Assume that
$\eta_A : F (A) \to G (A)$
is an isomorphism. Then $\eta$ is an isomorphism. 
\end{lem}

\begin{proof}
Suppose we are given a distinguished triangle
$M' \to M \to M'' \distri$
in $\til{\cat{D}}(\cat{DGMod} A)$, such that two of the three morphisms 
$\eta_{M'}$, $\eta_{M}$ and $\eta_{M''}$ are isomorphisms. Then the third is
also an isomorphism. 

Since both functors $F, G$ commute with shifts and direct sums, and since
$\eta_A$ is an isomorphism, it follows that 
$\eta_P$ is an isomorphism for any free DG $A$-module $P$. 

Next consider a semi-free DG module $P$. 
Choose any semi-basis $Z = \bigcup_{j \geq 0} Z_j$ of $P$. This gives rise to
an exhaustive ascending 
filtration $\{ P_j \}_{j \geq -1}$ of $P$ by DG submodules,
with $P_{-1} = 0$. For
every $j$ we have a distinguished triangle
\[ P_{j-1} \xar{\theta_j} P_j \to P_j / P_{j-1} \distri \]
in $\til{\cat{D}}(\cat{DGMod} A)$, where 
$\theta_j : P_{j-1} \to P_j$ is the inclusion. 
Since $P_j / P_{j-1}$ is a free DG module, by
induction we conclude that $\eta_{P_j}$ is an isomorphism for every $j$. 
The telescope construction (see \cite[Remark 2.2]{BN}) gives a distinguished
triangle
\[ \bigoplus_{j \in \N} P_j \xar{\ \Theta \ } \bigoplus_{j \in \N} P_j \to 
P  \distri , \]
with
\[ \Theta|_{P_{j-1}} := (\bsym{1}_{}, - \theta_j) : P_{j-1} \to
P_{j-1} \oplus P_{j} . \]
This shows that $\eta_P$ is an isomorphism. 

Finally, any DG module $M$ admits a quasi-isomorphism 
$P \to M$ with $P$ semi-free. Therefore $\eta_M$ is an isomorphism.
\end{proof}

Let $\cat{E}$ be a  be a full triangulated subcategory of 
$\til{\cat{D}}(\cat{DGMod} A)$ which is closed under infinite direct sums, and
let $M \in \cat{E}$.  Fix a K-projective resolution 
$P \to M$ in $\cat{DGMod}A$, and let 
$B := \opn{End}_A(P)$. So $B$ is a derived endomorphism DG algebra of $M$
(Definition \ref{dfn:310}).
Since $P \in \cat{DGMod} A \ot_{\K} B$, there is a triangulated functor 
\begin{equation} \label{eqn:450}
G :  \til{\cat{D}}(\cat{DGMod} B^{\mrm{op}}) \to
\til{\cat{D}}(\cat{DGMod} A) \ , \ 
G(N) := N \ot_B^{\mrm{L}} P
\end{equation}
which is calculated by K-flat resolutions in $\cat{DGMod} B^{\mrm{op}}$.
(Warning: $P$ is usually not K-flat over $B$.)
The functor $G$ commutes with infinite direct sums, and 
$G(B) \cong P \cong M$
in $\til{\cat{D}}(\cat{DGMod} A)$.
Therefore $G(N) \in \cat{E}$ for every
$N \in \til{\cat{D}}(\cat{DGMod} B^{\mrm{op}})$. 

Because $P$ is K-projective over $A$, there is a triangulated functor 
\begin{equation} \label{eqn:451}
F :  \til{\cat{D}}(\cat{DGMod} A) \to
\til{\cat{D}}(\cat{DGMod} B^{\mrm{op}})  \ , \ 
F(L) := \opn{Hom}_A(P, L) . 
\end{equation}
We have $F(M) \cong F(P) \cong B$ in 
$\til{\cat{D}}(\cat{DGMod} B^{\mrm{op}})$.

\begin{lem} \label{lem:450}
The functor 
$F|_{\cat{E}} : \cat{E} \to \til{\cat{D}}(\cat{DGMod} B^{\mrm{op}})$
commutes with infinite direct sums if and only if
$M$ is a compact object of $\cat{E}$.
\end{lem}

\begin{proof}
We know that 
\[ \opn{Hom}_{\til{\cat{D}}(\cat{DGMod} A)}(M, L[j]) \cong 
\mrm{H}^j (\opn{RHom}_{A}(M, L)) \cong 
\mrm{H}^j (F (L)) , \]
functorially for $L \in \til{\cat{D}}(\cat{DGMod} A)$. 
So $M$ is compact relative to $\cat{E}$ if and only if the functors
$\mrm{H}^j \circ F$ commute with direct sums in $\cat{E}$.
But that is the same as asking $F$ to commute with direct sums in $\cat{E}$.
\end{proof}

\begin{thm} \label{thm:34}
Let $A$ be a DG $\K$-algebra, let $\cat{E}$ be a  be a full triangulated
subcategory of $\til{\cat{D}}(\cat{DGMod} A)$ which is closed under infinite
direct sums, and let $M$ be a compact generator of $\cat{E}$. 
Choose a K-projective resolution $P \to M$ in $\cat{DGMod} A$, and define 
$B := \opn{End}_{A}(P)$. Then the functor 
\[ F|_{\cat{E}} : \cat{E} \to \til{\cat{D}}(\cat{DGMod} B^{\mrm{op}}) \]
from \tup{(\ref{eqn:451})} is an equivalence of triangulated categories, with 
quasi-inverse the functor $G$ from \tup{(\ref{eqn:450})} .
\end{thm}

\begin{proof}
Let us write $\cat{D}(A) := \til{\cat{D}}(\cat{DGMod} A)$ etc. 
We begin by proving that the functors $F$ and $G$ are adjoints. 
Take any $L \in \cat{D}(A)$ and $N \in \cat{D}(B^{\mrm{op}})$. We have to
construct a  bijection
\[ \opn{Hom}_{\cat{D}(A)}(G (N), L) \cong
\opn{Hom}_{\cat{D}(B^{\mrm{op}})}(N, F (L)) , \]
which is bifunctorial. Choose a K-projective resolution 
$Q \to N$ in $\cat{DGMod} B^{\mrm{op}}$. 
Since the DG $A$-module $Q \otimes_B P$ is K-projective, we have a sequence of
isomorphisms (of $\K$-modules)
\[ \begin{aligned}
& \opn{Hom}_{\cat{D}(A)}(G (N), L) \cong 
\mrm{H}^0 (\opn{RHom}_{A}( G (N), L)) \\
& \quad \cong \mrm{H}^0 (\opn{Hom}_{A}( Q \otimes_B P , L))
\cong \mrm{H}^0 \bigl( \opn{Hom}_{B^{\mrm{op}}}( Q , \opn{Hom}_A(P, L)) \bigr)
\\
& \quad \cong \mrm{H}^0 (\opn{RHom}_{B^{\mrm{op}}}( N , F (L) ) 
\cong \opn{Hom}_{\cat{D}(B^{\mrm{op}})}( N, F (L)) .
\end{aligned} \]
The only choice made was in the K-projective resolution $Q \to N$, so all is
bifunctorial. The corresponding morphisms 
$\bsym{1} \to F \circ G$ and $G \circ F \to \bsym{1}$ are denoted by $\eta$ and
$\zeta$ respectively. 

Next we will prove that $G$ is fully faithful. We do this by showing that
for every $N$ the morphism
$\eta_N :  N \to (F \circ G) (N)$
in $\cat{D}(B^{\mrm{op}})$ is an isomorphism. We know that $G$ factors via the
full subcategory $\cat{E} \subset \cat{D}(A)$, and therefore, using Lemma 
\ref{lem:450}, we know that the functor $F \circ G$ commutes with
infinite direct sums. So by Lemma \ref{lem:270} it suffices to check
for $N = B$. But in this case $\eta_B$ is the
canonical homomorphism of DG $B^{\mrm{op}}$-modules
$B \to \opn{Hom}_{A}(P, B \otimes_B P)$,
which is clearly bijective. 

It remains to prove that the essential image of the functor $G$ is
$\cat{E}$. Take any $L \in \cat{E}$, and 
consider the distinguished triangle
$(G \circ F)(L) \xar{\zeta_{L}} L \to L' \distri$
in $\cat{E}$, in which $L' \in \cat{E}$ is the mapping cone of $\zeta_{L}$.
Applying $F$ and
using $\eta$ we get a distinguished triangle
$F(L) \xar{1_{F(L)}} F(L) \to F(L') \distri$. 
Therefore $F(L') = 0$. But 
$\opn{RHom}_{A}(M, L') \cong  F(L')$,
and therefore 
$\opn{Hom}_{\cat{D}(A)}(M, L'[i]) = 0$
for every $i$. Since $M$ is a generator of $\cat{E}$ we get $L' = 0$. Hence 
$\zeta_{L}$ is an isomorphism, and so $L$ is in the essential
image of $G$.
\end{proof}

\section{The Main Theorem} \label{sec:ddc}

This is our interpretation of the completion appearing in Efimov's recent paper
\cite{Ef}, that is attributed to Kontsevich; cf.\ Remark \ref{rem:1} below
for a comparison to \cite{Ef} and to similar results in recent literature. 
Here is the setup for this section:
$A$ is a commutative ring, and $\a$ is a weakly proregular ideal in
$A$. We do not assume that $A$ is noetherian nor $\a$-adically complete. Let
$\what{A} := \La_{\a}(A)$, the $\a$-adic completion of $A$, and let 
$\what{\a} := \a \cdot \what{A}$, which is an ideal of $\what{A}$. Since the
ideal $\a$ is finitely generated, it follows that the $A$-module $\what{A}$
is $\a$-adically complete, and hence as a ring $\what{A}$ is
$\what{\a}$-adically complete.

The full subcategory 
$\cat{D}(\cat{Mod} A)_{\a \tup{-tor}} \subset \cat{D}(\cat{Mod} A)$
is triangulated and closed under infinite direct sums. 
The results of Sections \ref{sec:der-end} and \ref{sec:der-morita} are invoked
with $\K := A$. 

Recall the Koszul complex $\mrm{K}(A; \bsym{a})$ associated to a finite sequence
$\bsym{a}$ in $A$; see Section \ref{sec:kosz}. It is a bounded complex of free
$A$-modules, and hence it is a K-projective DG $A$-module. 
The next result was proved by several authors (see 
\cite[Proposition 6.1]{BN}, \cite[Corollary 5.7.1(ii)]{LN} and 
\cite[Proposition 6.6]{Ro}).

\begin{prop}
Let $\bsym{a}$ be a finite sequence that generates $\a$. 
Then the Koszul complex $\opn{K}(A; \bsym{a})$ is a compact generator of 
$\cat{D}(\cat{Mod} A)_{\a \tup{-tor}}$.
\end{prop}

Of course there are other compact generators of 
$\cat{D}(\cat{Mod} A)_{\a \tup{-tor}}$.

\begin{thm} \label{thm:22}
Let $A$ be a commutative ring, let  $\a$ be a weakly proregular ideal in $A$,
and let $M$ be a compact generator of $\cat{D}(\cat{Mod} A)_{\a \tup{-tor}}$.
Choose some K-projective resolution $P \to M$ in 
$\cat{C}(\cat{Mod} A)$, and let 
$B :=  \opn{End}_A(P)$.  Then
$\opn{Ext}^i_{B}(P) = 0$ for all $i \neq 0$, and there is a unique isomorphism
of $A$-algebras
$\opn{Ext}^0_{B}(P) \cong \what{A}$.
\end{thm}

Recall that the DG $A$-algebra $B$ is a derived endomorphism DG algebra of $M$
(Definition \ref{dfn:310}), and the graded $A$-algebra 
$\opn{Ext}_{B}(P)$ is a derived double centralizer of $M$ (Definition
\ref{dfn:303}).

We need a few lemmas before proving the theorem.

\begin{lem} \label{lem:21}
Let $M$ be a compact object of $\cat{D}(\cat{Mod} A)_{\a \tup{-tor}}$.
Then $M$ is also compact in $\cat{D}(\cat{Mod} A)$, so it is a perfect complex
of $A$-modules. 
\end{lem}

\begin{proof}
Choose a finite sequence $\bsym{a}$ that generates $\a$. By
\cite[Corollary 4.26]{PSY} there is an isomorphism of functors 
$\mrm{R} \Gamma_{\a} \cong \opn{K}^{\vee}_{\infty}(A; \bsym{a}) \ot_A -$,
where $\opn{K}^{\vee}_{\infty}(A; \bsym{a})$
is the infinite dual Koszul complex. Therefore the functor 
$\mrm{R} \Gamma_{\a}$ commutes with infinite direct sums.

Let $N \in \cat{D}(\cat{Mod} A)$, and consider the function 
\[ \opn{Hom}(\bsym{1}, \sigma^{\mrm{R}}_N) : 
\opn{Hom}_{\cat{D}(\cat{Mod} A)}(M, \mrm{R} \Gamma_{\a} (N))
\to \opn{Hom}_{\cat{D}(\cat{Mod} A)}(M, N) . \]
Given a morphism $\alpha : M \to N$ in $\cat{D}(\cat{Mod} A)$
define 
\[ \beta := \mrm{R} \Gamma_{\a}(\alpha) \circ 
(\sigma^{\mrm{R}}_M)^{-1} : M \to \mrm{R} \Gamma_{\a} (N) . \]
Since the functor $\mrm{R} \Gamma_{\a}$ is idempotent (Theorem
\ref{thm:201}(1)), the function $\alpha \mapsto \beta$ is an inverse to 
$\opn{Hom}(\bsym{1}, \sigma^{\mrm{R}}_N)$, so the latter is bijective. 

Let $\{ N_z \}_{z \in Z}$ be a collection of objects of 
$\cat{D}(\cat{Mod} A)$. Due to the fact that $M$ is a compact object of 
$\cat{D}(\cat{Mod} A)_{\a \tup{-tor}}$, and to the observations above, 
we get isomorphisms
\[ \begin{aligned}
& \bigoplus_z \, \opn{Hom}_{\cat{D}(\cat{Mod} A)}(M, N_z) 
\cong \bigoplus_z \, 
\opn{Hom}_{\cat{D}(\cat{Mod} A)}(M, \mrm{R} \Gamma_{\a} (N_z)) \\
& \quad \cong \opn{Hom}_{\cat{D}(\cat{Mod} A)} 
\bigl( M, \boplus_z \, \mrm{R} \Gamma_{\a} (N_z)
\bigr) 
\cong \opn{Hom}_{\cat{D}(\cat{Mod} A)} \bigl( M, \mrm{R} \Gamma_{\a} \bigl(
\boplus_z \,  N_z \bigr) \bigr) \\
& \quad \cong \opn{Hom}_{\cat{D}(\cat{Mod} A)} \bigl( M, \boplus_z \,  
N_z \bigr) .
\end{aligned} \]
We see that $M$ is also compact in $\cat{D}(\cat{Mod} A)$.
\end{proof}

Consider the contravariant functor
\[ D : \cat{D}(\cat{Mod} B) \to \cat{D}(\cat{Mod} B^{\mrm{op}}) \]
defined by choosing an injective resolution $A \to I$ over $A$, and
letting
$D := \opn{Hom}_A(-, I)$.

\begin{lem} \label{lem:22}
The functor $D$ induces a duality \tup{(}i.e.\ a contravariant
equivalence\tup{)} between the full subcategory of 
$\cat{D}(\cat{Mod} B)$ consisting of objects perfect over $A$, and the full
subcategory
of $\cat{D}(\cat{Mod} B^{\mrm{op}})$ consisting of objects perfect over $A$.
\end{lem}

\begin{proof}
Take $M \in \cat{D}(\cat{Mod} B)$ which is perfect over $A$. It is enough to
show that the canonical homomorphism of DG $B$-modules
\begin{equation} \label{eqn:4}
M \to (D \circ D) (M) = \opn{Hom}_A(\opn{Hom}_A(M, I), I)
\end{equation}
is a quasi-isomorphism. For this we can forget the $B$-module structure, and
just view this as a homomorphism of DG $A$-modules. Choose a resolution 
$P \to M$ where $P$ is a bounded complex of finitely generated projective
$A$-modules. We can replace $M$ with $P$ in equation (\ref{eqn:4}); and after
that we can replace $I$ with $A$; now it
is clear that this is a quasi-isomorphism.
\end{proof}

\begin{lem} \label{lem:305}
Let $M$ and $N$ be K-flat complexes of $A$-modules.
We write $\what{M} :=  \La_{\a}(M)$ and $\what{N} :=  \La_{\a}(N)$.
\begin{enumerate}
\item The morphisms $\xi_M : \mrm{L} \La_{\a}(M) \to \what{M}$
and 
$\tau^{\mrm{L}}_{\what{M}} : \what{M} \to \mrm{L} \La_{\a}(\what{M})$
are isomorphisms. 

\item The homomorphism 
\[ \opn{Hom}(\tau_M, \bsym{1}) :
\opn{Hom}_{\cat{D}(\cat{Mod} A)}( \what{M}, \what{N}) 
\to 
\opn{Hom}_{\cat{D}(\cat{Mod} A)}( M, \what{N}) \]
is bijective.
\end{enumerate} 
\end{lem}

\begin{proof}
(1) The morphism $\xi_M$ is an isomorphism by \cite[Proposition 3.6]{PSY}.
By Theorem \ref{thm:201}(1) the complex $\mrm{L} \La_{\a}(M)$ is
cohomologically complete; and therefore $\what{M}$ is also 
cohomologically complete. But this means that $\tau^{\mrm{L}}_{\what{M}}$ is an
isomorphism.

\medskip \noindent
(2) Take a morphism $\alpha : M \to \what{N}$ in $\cat{D}(\cat{Mod} A)$. 
By part (1) we know that $\xi_M$ and $\tau^{\mrm{L}}_{\what{N}}$ are
isomorphisms, so we can define 
\[ \beta 
:= (\tau^{\mrm{L}}_{\what{N}})^{-1} \circ \mrm{L} \La_{\a}(\alpha)  
\circ \xi_M^{-1} : \what{M} \to \what{N} . \]
The function $\alpha \mapsto \beta$ is an inverse to 
$\opn{Hom}(\tau_M, \bsym{1})$.
\end{proof}

\begin{proof}[Proof of Theorem \tup{\ref{thm:22}}]
We shall calculate $\opn{Ext}_{B}(P)$ indirectly. 

By Lemma \ref{lem:21} we know that $M$, and hence also $P$, is perfect over $A$.
So according to Lemma \ref{lem:22} there is an isomorphism of graded
$A$-algebras
\[ \opn{Ext}_{B}(P) \cong \opn{Ext}_{B^{\mrm{op}}}(D (P))^{\mrm{op}} . \]
Next we note that 
\[ D (P) = \opn{Hom}_A(P, I) \cong \opn{Hom}_A(P, A) = F (A) \] 
in $\til{\cat{D}}(\cat{DGMod} B^{\mrm{op}})$. Here 
$F$ is the functor from (\ref{eqn:451}).
Therefore we get an isomorphism of graded $A$-algebras
\[ \opn{Ext}_{B^{\mrm{op}}}(D (P)) \cong 
\opn{Ext}_{B^{\mrm{op}}}(F (A)) . \]

Let 
$N := \mrm{R} \Gamma_{\a} (A) \in \cat{D}(\cat{Mod} A)$.
We claim that $F (A) \cong  F (N)$ in $\til{\cat{D}}(\cat{DGMod} B^{\mrm{op}})$.
To see this, we first note that the canonical morphism 
$\sigma^{\mrm{R}}_A : N \to A$ in $\cat{D}(\cat{Mod} A)$ can be represented by
an actual DG module homomorphism
$N \to A$ (say by replacing $N$ with a K-projective resolution of it). Consider
the induced homomorphism 
$\opn{Hom}_A(P, N) \to \opn{Hom}_A(P, A)$
of DG $B^{\mrm{op}}$-modules. Like in the proof of Lemma \ref{lem:22}, it
suffices to show that this is a quasi-isomorphism of DG $A$-modules. This is
true since, by {\em GM Duality} \cite[Theorem 7.12]{PSY},
 the canonical morphism
\[ \opn{RHom}(\bsym{1}, \sigma^{\mrm{R}}_A) : 
\opn{RHom}_A(M, N) \to \opn{RHom}_A(M, A) \]
in $\cat{D}(\cat{Mod} A)$ is an isomorphism.
We conclude that there is a graded $A$-algebra isomorphism
\[ \opn{Ext}_{B^{\mrm{op}}}(F (A)) \cong 
\opn{Ext}_{B^{\mrm{op}}}(F (N)) . \]

Take $\cat{E} := \cat{D}(\cat{Mod} A)_{\a \tup{-tor}}$ in Theorem \ref{thm:34}.
Since 
\[ F|_{\cat{E}} : \cat{E} \to \til{\cat{D}}(\cat{DGMod} B^{\mrm{op}}) \]
is an equivalence, and $\cat{E}$ is full in 
$\cat{D}(\cat{Mod} A)$, we see that $F$
induces  an isomorphism of graded $A$-algebras 
\[ \opn{Ext}_{A}(N) \cong \opn{Ext}_{B^{\mrm{op}}}(F(N))  .  \]

The next step is to use the MGM equivalence.
We know that 
$\mrm{L} \Lambda_{\a} (N) \cong \what{A}$
in $\cat{D}(\cat{Mod} A)$. And the functor $\mrm{L} \Lambda_{\a}$ induces an
isomorphism of graded $A$-algebras
\[ \opn{Ext}_{A}(N) \cong \opn{Ext}_{A}(\what{A}) . \]

It remains to analyze the graded $A$-algebra 
$\opn{Ext}_{A}(\what{A})$.  
By Lemma \ref{lem:305} the homomorphism
\[ \opn{Hom}(\tau_A, \bsym{1}) : 
\opn{Hom}_{\cat{D}(\cat{Mod} A)} ( \what{A} , \what{A}[i] ) \to  
\opn{RHom}_{\cat{D}(\cat{Mod} A)} ( A, \what{A}[i] ) \]
is bijective for every $i$. Therefore  
$\opn{Ext}^i_{A}(\what{A}) = 0$ for $i \neq 0$, 
and the $A$-algebra homomorphism
$\what{A} \to \opn{Ext}^0_{A}(\what{A})$
is bijective. 

Combining all the steps above we see that 
$\opn{Ext}^i_{B}(P) = 0$ for $i \neq 0$, and there is an $A$-algebra
isomorphism 
$\opn{Ext}^0_{B}(P) \cong \what{A}^{\mrm{op}}$. 
But $\what{A}$ is commutative, so 
$\what{A}^{\mrm{op}} = \what{A}$.

Regarding the uniqueness: since the image of the ring homomorphism 
$A \to \what{A}$ is dense, and $\what{A}$ is $\what{\a}$-adically complete, it
follows that the only $A$-algebra automorphism of $\what{A}$ is the identity.
Therefore the $A$-algebra isomorphism 
$\opn{Ext}_{B}^0(P) \cong \what{A}$ that we produced is unique.
\end{proof}

\begin{rem}
To explain how surprising this theorem is, take the case 
$P = M := \mrm{K}(A; \bsym{a})$, the Koszul complex associated to a sequence
$\bsym{a} = (a_1, \ldots, a_n)$ that generates the ideal $\a$.

As a free $A$-module (forgetting the grading and the differential), we have
$P \cong A^{n^2}$. The grading of $P$ depends on $n$ only (it is an exterior
algebra). The differential of $P$ is the only place where the sequence
$\bsym{a}$ enters. Similarly, the DG algebra $B = \opn{End}_A(P)$ is a graded
matrix algebra over $A$, of size $n^2 \times n^2$. The differential of $B$ is
where $\bsym{a}$ is expressed. 

Forgetting the differentials, i.e.\ working with the graded $A$-module
$P^{\natural}$, classical Morita theory tells us that 
$\opn{End}_{B^{\natural}}(P^{\natural}) \cong A$
as graded $A$-algebras. Furthermore,  $P^{\natural}$ is a projective
$B^{\natural}$-module, so we even have 
$\opn{Ext}_{B^{\natural}}(P^{\natural}) \cong A$.

However, the theorem tells us that for the DG-module structure of $P$ we
have $\opn{Ext}_{B}(P) \cong \what{A}$.
Thus we get a transcendental outcome -- the completion $\what{A}$ -- 
by a homological operation with finite input (basically finite linear algebra
over $A$ together with a differential). 
\end{rem}

\begin{rem} \label{rem:1}
Our motivation to work on completion by derived double centralizer came from
looking at the recent paper \cite{Ef} by Efimov. The main result of \cite{Ef}
is Theorem 1.1 about the completion of the category 
$\cat{D}(\cat{QCoh} X)$ of a noetherian scheme $X$ along a closed subscheme
$Y$. This idea is attributed to Kontsevich. Corollary 1.2 of \cite{Ef} 
is a special case of our Theorem \ref{thm:22}: it has the extra assumptions
that the ring $A$ is noetherian and regular (i.e.\ it has finite global
cohomological
dimension). 

After writing the first version of our paper, we learned that a similar result
was proved by Dwyer-Greenlees-Iyengar \cite{DGI}. In that paper the authors
continue the work of \cite{DG} on derived completion and torsion. Their main
result is Theorem 4.10, which is a combination of  MGM equivalence and
derived Morita equivalence in an abstract setup (that includes algebra and
topology). The manifestation of this main result in commutative algebra is 
\cite[Proposition 4.20]{DGI}, that is also a special case of our 
Theorem \ref{thm:22}: the ring $A$ is noetherian, and the quotient ring
$A / \a$ is regular. 

Recall that our Theorem \ref{thm:22} only requires the ideal $\a$ to be weakly
proregular, and there is no regularity condition on the rings $A$ and $A / \a$
(the word ``regular'' has a double meaning here!).
It is quite possible that the methods of \cite{DGI} or \cite{Ef} can be pushed
further to remove the regularity conditions from the rings $A$ and $A / \a$.
However, it is less likely that these methods can handle the non-noetherian case
(i.e.\ assuming only that the ideal $\a$ is weakly proregular).
\end{rem}


\end{document}